\documentclass[a4paper,11pt,french,english]{amsart}
\usepackage[english]{babel}
\usepackage{lmodern}
\usepackage[T1]{fontenc}
\usepackage{amsmath,amsthm,amssymb,amsfonts}
\usepackage{amscd}
\usepackage[a4paper,left=3.8cm,right=3.8cm,top=4cm,bottom=4cm]{geometry}
\usepackage{enumitem}
\usepackage{fancyhdr}
\usepackage[linktocpage=true]{hyperref} \usepackage{url} \usepackage{csquotes}

\newcommand{\Sy}{\mathrm{Sym}(\Omega)}

\newcommand{\autd}{\mathrm{Aut}(T_\Omega)}
\newcommand{\aut}{\mathrm{Aut}(T)}

\newcommand{\sg}{\sigma(g,v)}

\newcommand{\Pg}{\mathrm{Pw}(G)}

\title{$C^\ast$-simplicity and the amenable radical}
\author{Adrien Le Boudec}
\address{Laboratoire de Math\'{e}matiques, Universit\'{e} Paris-Sud 11, 91405 Orsay, France}
\address{UCLouvain, IRMP, Chemin du Cyclotron 2, 1348 Louvain-la-Neuve, Belgium}
\email{adrien.leboudec@uclouvain.be}
\date{\today}

\theoremstyle{plain}
\newtheorem{thm}{Theorem}[section]
\newtheorem{prop}[thm]{Proposition}
\newtheorem{cor}[thm]{Corollary}
\newtheorem{lem}[thm]{Lemma}

\newtheorem{thmintro}{Theorem}

\theoremstyle{definition}
\newtheorem{defi}[thm]{Definition}

\newtheorem{rmq}[thm]{Remark}

\begin{document}

\maketitle

\begin{abstract}
A countable group is $C^\ast$-simple if its reduced $C^\ast$-algebra is simple. It is well-known that $C^\ast$-simplicity implies that the amenable radical of the group must be trivial. We show that the converse does not hold by constructing explicit counter-examples. We additionally prove that every countable group embeds into a countable group with trivial amenable radical and that is not $C^\ast$-simple. 
\end{abstract}

\section{Introduction}

If $G$ is a discrete countable group, the reduced $C^\ast$-algebra of $G$ is the operator norm closure of the group algebra $\mathbf{C}[G]$ acting by left-regular representation on the Hilbert space $\ell^2(G)$. We say that $G$ is $C^\ast$-simple if its reduced $C^\ast$-algebra has no non-trivial two-sided ideal. This is equivalent to saying that any unitary representation of $G$ that is weakly contained in the regular representation $\lambda_G$, is actually weakly equivalent to $\lambda_G$. For a proof of this equivalence and complements about $C^\ast$-simplicity, we refer the reader to \cite{dlH-survey}.

\medskip

The study of the class of $C^\ast$-simple groups has been of central interest since the proof of $C^\ast$-simplicity of the free group of rank two \footnote{We learned from Pierre de la Harpe that, although the article \cite{Pow} was published in 1975, the proof of the $C^\ast$-simplicity of $\mathbb{F}_2$ was actually obtained by Powers in 1968.} \cite{Pow}. $C^\ast$-simplicity has been extensively generalized to many classes of groups, among other non-trivial free products \cite{Pas-Sal} (see also \cite{Ake-Lee,Bed-disgr,dlH-P}), Gromov-hyperbolic groups \cite{Har-88} and relatively hyperbolic groups \cite{A-M}, lattices in semi-simple connected Lie groups \cite{B-C-dlH}, or centerless mapping class groups and outer automorphism groups of free groups \cite{Br-dlH}. It was proved in \cite{DGO} that the methods from \cite{Ake-Lee} actually apply to any acylindrically hyperbolic group (therefore unifying the aforementioned results from \cite{Pas-Sal,Har-88,A-M,Br-dlH}). More recently, $C^\ast$-simplicity has also been obtained for free Burnside groups of large odd exponent \cite{O-O} and some Tarski monsters \cite{KK,BKKO}.

\medskip

It has been known for a long time that the existence of a non-trivial amenable normal subgroup is an obstruction to $C^\ast$-simplicity \cite{Pas-Sal}. Therefore if $G$ is $C^\ast$-simple, then the amenable radical of $G$, i.e.\ the largest amenable normal subgroup of $G$, must be trivial. It was considered as a major problem to determine whether triviality of the amenable radical is always equivalent to $C^\ast$-simplicity. This problem is for instance discussed in \cite{B-dlH} and in the survey \cite{dlH-survey}[Question 4]. It has been answered positively for the class of linear groups \cite{Poz,BKKO}, and the recent work \cite{BKKO} gives a positive answer as well for the class of groups having only countably many amenable subgroups. The main purpose of this paper is to show that triviality of the amenable radical does not imply $C^\ast$-simplicity in general, therefore answering a long-standing open question. 

\medskip

Problems related to $C^\ast$-simplicity recently experienced major advances. A powerful dynamical approach has been initiated in \cite{KK}. This strategy has been further elaborated in \cite{BKKO}, providing new proofs of $C^\ast$-simplicity of many already known examples. In the realm of locally compact groups, the first construction of non-discrete $C^\ast$-simple groups has recently been carried out in \cite{SvR}.

\medskip

The work \cite{BKKO} also settled the long-standing open problem of characterizing discrete countable groups with the unique trace property (a definition of which can be found in \cite{dlH-survey}), as those for which the amenable radical is trivial. In particular $C^\ast$-simplicity implies the unique trace property. Combined with \cite{BKKO}, our result implies that the converse does not hold, thus filling the last gap in determining the implications between $C^\ast$-simplicity, unique trace property and triviality of the amenable radical. 

\medskip

One of the main results of \cite{KK} characterizes $C^\ast$-simple groups as those having a topologically free boundary action. A partial converse has been proved in \cite{BKKO}, namely that for $C^\ast$-simple groups, \textit{any} boundary action with amenable stabilizers must be topologically free. This last result is an essential argument in our first proof of Theorem \ref{thm-general-c*-intro} below. Our second proof relies on a striking argument from \cite{H-O}, which provides some unitary representation of the group which cannot weakly contain the left-regular representation. 

\section{Results}

The groups for which we will show the non-simplicity of the reduced $C^\ast$-algebra are defined in terms of an action on a tree. In all the paper, $T$ will denote a (not necessarily locally finite) tree. We refer to the beginning of Section \ref{sec-general-result} for the definitions of the terms appearing in the next theorem.

\begin{thmintro} \label{thm-general-c*-intro}
Let $G \leq \aut$ be a countable group, whose action on $T$ is minimal and of general type. Assume that fixators of half-trees in $G$ are non-trivial, and that there is some $\xi \in \partial T$ such that $G_{\xi}$ is amenable. Then $G$ has trivial amenable radical and is not $C^\ast$-simple. 
\end{thmintro}

We emphasize that in the above theorem, we do consider $G$ as a discrete group, although a group satisfying these assumptions \textit{cannot} be a discrete subgroup of the topological group $\aut$. More generally, unless specified otherwise, in this paper all groups are viewed as discrete groups. In particular when referring to amenability, it is always with respect to the discrete topology. 

\subsection*{Piecewise prescribed tree automorphisms}

Our first illustration of Theorem \ref{thm-general-c*-intro} comes from the following construction. Given a subgroup $G \leq \aut$, we introduce the group $\Pg \leq \aut$ of automorphisms of $T$ acting piecewise like $G$ (see Section \ref{sec-Pw} for a formal definition). Constructions of the same flavor had already been considered \textit{at the level of the boundary}: for $\aut$ this gives rise to the notion of almost automorphism of $T$, and for $\mathrm{PSL}(2,\mathbf{R})$ acting on the hyperbolic disc we obtain the group of piecewise projective homeomorphisms of the circle, recently considered in \cite{Monod}. We point out that this construction is different in the sense that $\Pg$ does act on $T$, and not only on its boundary. Note also that this \enquote{inner piecewise-ation} process applied to $\mathrm{PSL}(2,\mathbf{R})$ would only give $\mathrm{PSL}(2,\mathbf{R})$ itself. 

\medskip

The following result shows that, under weak assumptions on the group $G$, the family of groups acting on $T$ piecewise like $G$ provides examples of countable groups that are not $C^\ast$-simple.

\begin{thmintro} \label{thm-pw-c*simple}
Let $G \leq \aut$ be a countable group whose action on $T$ is minimal and of general type, and such that fixators of vertices in $G$ are amenable. Let $\Gamma$ be a subgroup of $\Pg$ that contains $G$ (and therefore has trivial amenable radical). 
\begin{enumerate}[label=(\alph*)] 
\item If fixators of half-trees in $\Gamma$ are non-trivial, then $\Gamma$ is not $C^\ast$-simple.
\item Assume that stabilizers of vertices in $G$ are non-trivial. Then $\Pg$ is not $C^\ast$-simple.
\end{enumerate}
\end{thmintro}

Considering for $G$ amalgamated products or HNN-extensions acting on their Bass-Serre tree, we thereby obtain a multitude of examples of countable non-$C^\ast$-simple groups with trivial amenable radical. We refer to Section \ref{sec-Pw} for details and examples.

\subsection*{Groups with prescribed local action}

We further illustrate Theorem \ref{thm-general-c*-intro} in the following way. Let $\Omega$ be a (possibly finite) countable set, and let $F \leq F' \leq \Sy$ be permutation groups on $\Omega$. If $T_\Omega$ is a regular tree of degree the cardinality of $\Omega$, we denote by $G(F,F')$ the subgroup of $\autd$ consisting of automorphisms whose local action is prescribed by $F'$ around all vertices, and by $F$ around all but finitely many vertices. We refer to Section \ref{sec-G(F,F')} for a formal definition.

\begin{thmintro} \label{thm-c*-g(f,f')}
Let $\Omega$ be a countable set, and let $F \lneq F' \leq \Sy$ be countable permutation groups such that $F$ acts freely on $\Omega$, $F'$ preserves the orbits of $F$ and has all its point stabilizers amenable. Then the countable group $G(F,F')$ has trivial amenable radical and is not $C^\ast$-simple.
\end{thmintro}

When $\Omega$ has finite cardinality, the groups $G(F,F')$ from Theorem \ref{thm-c*-g(f,f')} give concrete and uncomplicated examples of non-$C^\ast$-simple groups with trivial amenable radical. They satisfy many additional interesting properties: they are finitely generated (but not finitely presented), they have the Haagerup property, and their asymptotic dimension is equal to one (which is the smallest possible asymptotic dimension for infinite countable groups). When $F$ acts transitively on $\Omega$, the subgroup $G(F,F')^\ast$ of index two in $G(F,F')$ consisting of automorphisms preserving the natural bipartition of vertices of $T_\Omega$, splits as an amalgamated product $A \ast_C B$, where $A \simeq B$ are infinite locally finite groups. Moreover there are natural permutation groups $F \leq F'$ such that $G(F,F')^\ast$ is simple, e.g.\ $F$ generated by a cycle of length $d$ and $F' = \mathrm{Alt}(d)$ for $d \geq 5$ odd. For proofs of these properties and for complements, we refer to \cite{LB-ae}.

When the set $\Omega$ is infinite, this construction is extremely flexible. It allows us to prove the following result.

\begin{thmintro} \label{thm-intro-embedding}
Every countable group $\Gamma$ embeds into a countable group $G$ with trivial amenable radical and that is not $C^\ast$-simple. If moreover $\Gamma$ is finitely generated or torsion free (or both), then so is $G$. 
\end{thmintro}

\subsection*{Acknowledgments}

I warmly thank Emmanuel Breuillard for explaining the work \cite{BKKO}, for stimulating discussions and for mentioning the question of $C^\ast$-simplicity of lattices in product of trees. I would also like to thank the organizers of the conference \textit{Geometries in Action} in Lyon in honor of \'{E}tienne Ghys, during which part of the interactions with Emmanuel Breuillard took place. Finally I also thank Yves de Cornulier and Pierre de la Harpe for their useful remarks.

\section{General result} \label{sec-general-result}

Throughout the paper, $T$ will be a simplicial tree that is \textit{not} assumed to be locally finite. We call a subtree of $T$ a half-tree if it is one of the two components obtained when removing an edge in $T$. If $A$ is a subtree of $T$ and $G$ is a group acting on $T$, the stabilizer of $A$ in $G$ is the set of elements of $G$ preserving $A$ setwise, and the fixator $G_A$ of $A$ is the set of elements of $G$ fixing pointwise $A$. We say that the action of $G$ on $T$ is minimal if there is no proper non-empty $G$-invariant subtree.

Recall that $g \in \aut$ is called hyperbolic if there is a bi-infinite geodesic line, called the axis of $g$, on which $g$ acts by translation. When this holds, $g$ has exactly two fixed points in $\partial T$. We say that the action of a group $G$ on $T$ is of general type if there exist two hyperbolic elements in $G$ without common fixed points in $\partial T$. We point out that the terminology strongly hyperbolic is sometimes used, see \cite{dlH-survey}. In this situation, the ping-pong argument applies and yields non-abelian free subgroups in $G$ \cite[p.\ 152]{Pa-Va}. Classical results about isometric group actions on trees assert that if the action of $G$ on $T$ is not of general type, then one of the following happens: $G$ stabilizes a vertex or an edge, or $G$ has a unique fixed pair or a unique fixed point in $\partial T$ \cite[Propositions 1 \& 2]{Pa-Va}.

\subsection{Construction of a boundary action}

Recall that if $G$ is a countable group, a $G$-boundary is a compact space $X$ endowed with an action of $G$ by homeomorphisms, such that every $G$-orbit in $X$ is dense, and that is strongly proximal, that is every probability measure on $X$ has a Dirac measure in the weak-closure of its $G$-orbit. The aim of this paragraph is to explain how to construct a $G$-boundary starting from an action of $G$ on $T$. 

\bigskip

If $v$ is a vertex of $T$ and $x \in T \cup \partial T$ is either a vertex or an end, there exists a unique geodesic from $v$ to $x$, that will be denoted $\left[v,x\right]$. If $v_1,\ldots,v_n$ are neighbours of the vertex $v$, we denote by $U(v;v_1,\ldots,v_n)$ the set of $x \in T \cup \partial T$ such that $\left[v,x\right]$ contains none of the vertices $v_1,\ldots,v_n$. Equipped with the topology generated by all the subsets $U(v;v_1,\ldots,v_n)$, where $v$ ranges over the set of vertices and $n \geq 0$, the set $T \cup \partial T$ is a compact space (see \cite[Proposition 4.2]{Mon-Sha}). We are grateful to Pierre-Emmanuel Caprace for pointing out the reference \cite{Mon-Sha} to our attention.

The action of the group $\mathrm{Aut}(T)$ on the tree $T$ extends to an action on $T \cup \partial T$ by homeomorphisms. Clearly $\mathrm{Aut}(T)$ preserves the (open) set $T_f$ of vertices of $T$ of finite degree, so that we also have an action of $\mathrm{Aut}(T)$ on the space $X = (T \setminus T_f) \cup \partial T$.

\begin{prop} \label{prop-construction-boundary}
Let $G \leq \aut$ be a countable group, whose action on $T$ is minimal and of general type. Then $X = (T \setminus T_f) \cup \partial T$ is a $G$-boundary. 
\end{prop}

\begin{proof}
We let $C$ be a closed non-empty $G$-invariant subset of $X$, and we prove that $C=X$. Remark that if $g \in G$ is hyperbolic, then $(g^n x)$ converges to the attracting fixed point of $g$ for every $x$ different from the repelling fixed point of $g$. Combining this observation with the fact that no point of $X$ can be fixed by all the hyperbolic elements of $G$ (because the action of $G$ is of general type), we see that $C$ must contain all fixed point of hyperbolic elements of $G$.

We claim that this implies that $C=X$. To see this, let us first prove that every end of $T$ lies in the closure of the set of fixed points of hyperbolic elements of $G$. To prove this, it is enough to show that every half-tree intersects the axis of some hyperbolic element. Argue by contradiction and assume that this is not true. Then the union of the axes of the hyperbolic elements of $G$ is therefore contained in a proper subtree of $T$. But it is classical that the latter is a $G$-invariant subtree, so we obtain a contradiction with the minimality of the action of $G$ on $T$. Therefore we have proved that $C$ must contain $\partial T$. The latter being dense in $X$ \cite[Proposition 4.4 (vi)]{Mon-Sha}, one must have $C=X$.

The fact that the action of $G$ on $X$ is strongly proximal is obtained similarly, and we refer to \cite[Example 2]{Ozawa-notes}. 
\end{proof}

\subsection{Proof of Theorem A}

We now give two different proofs of Theorem \ref{thm-general-c*-intro} from the introduction.

\begin{proof}[Proof 1]
The fact that $G$ has no non-trivial amenable normal subgroup is classical, we repeat the argument for completeness. We actually prove that if $N$ is a normal subgroup of $G$ not containing non-abelian free subgroups, then $N$ must be trivial. Since it does not contain non-abelian free subgroups, the group $N$ must stabilize a vertex or an edge, or must have a unique finite orbit in $\partial T$. In the latter case, this finite orbit has cardinality one or two and must be $G$-invariant since $N$ is normal in $G$, which is impossible since the action of $G$ is of general type. Similarly if $N$ stabilizes an edge without fixing a vertex, then this edge has to be $G$-invariant, which contradicts the existence of hyperbolic elements in $G$. So the set of vertices of $T$ fixed by $N$ is a non-empty subtree, and by $G$-invariance and minimality of the action of $G$ we obtain that it must be the entire tree $T$, which exactly means that $N$ is trivial.

We now turn to the proof that $G$ is not $C^\ast$-simple. Since the action of $G$ on $T$ is minimal and of general type, the set $X = (T \setminus T_f) \cup \partial T$ is a $G$-boundary by Proposition \ref{prop-construction-boundary}. Fixators of half-trees in $G$ are non-trivial by assumption, so by definition of the topology this implies that the action of $G$ on $X$ is not topologically free. Moreover there is a point stabilizer $G_{\xi}$ that is assumed to be amenable, so according to \cite[Proposition 2.6]{BKKO}, the group $G$ cannot be $C^\ast$-simple.
\end{proof}

We now give a second and independent proof of Theorem \ref{thm-general-c*-intro}. 

\begin{proof}[Proof 2]
As proved in \cite{H-O}, the existence of an amenable subgroup $H \leq G$ together with non-trivial elements $a,b \in G$ having disjoint support in $G / H$ imply that $G$ cannot be $C^\ast$-simple. Indeed, an easy computation shows that since $a,b$ have disjoint support in $G/H$, the element $(1-a)(1-b)$ acts trivially by convolution on $\ell^2(G/H)$. Since $H$ is amenable, the quasi-regular representation $\lambda_{G/H}$ must be weakly contained in $\lambda_G$, while the converse is impossible by the previous observation. So $G$ is not $C^\ast$-simple.

We take $H = G_{\xi}$, so that $G/H$ can be identified with the $G$-orbit of $\xi$ in $\partial T$. Consider an edge $e$ of $T$, and let $T_1$ and $T_2$ be the two half-trees emanating from $e$. The partition $\partial T = \partial T_1 \sqcup \partial T_2$ yields a partition of $G \cdot \xi$ into two non-empty subsets. By assumption fixators of half-trees in $G$ are not reduced to the identity, so we may plainly find non-trivial elements $a \in G_{T_1}$ and $b \in G_{T_2}$, and these have disjoint support in $G \cdot \xi$ by construction. This proves the statement.
\end{proof}

\subsection{Product of trees}

The groups $G(F,F')$ that will be shown not be $C^\ast$-simple in Section \ref{sec-G(F,F')} are connected to Burger-Mozes' finitely presented torsion free simple groups constructed as lattices in the product of two trees \cite{BM-IHES-2}. Therefore this naturally raises the question whether these groups are $C^\ast$-simple. The following result, which relies on \cite{BKKO}, shows that this is indeed the case. We mention that the $C^\ast$-simplicity of these groups is also obtained in \cite{Kar-Sag}. 

\begin{prop} \label{prop-product-trees}
Let $\Gamma  \leq \mathrm{Aut}(T_1) \times \ldots \times \mathrm{Aut}(T_n)$ be a discrete subgroup, such that every locally finite subgroup of $\Gamma$ is finite. Then $\Gamma$ is $C^\ast$-simple if and only if $\Gamma$ has trivial amenable radical. 
\end{prop}

\begin{proof}
According to Theorem 3.8 in \cite{BKKO}, the equivalence between triviality of the amenable radical and $C^\ast$-simplicity holds for the class of countable groups having only countably many amenable subgroups, so it enough to show that $\Gamma$ has this property. Upon replacing $\Gamma$ by a finite index subgroup, we may assume that $\Gamma$ acts without inversion on each $T_i$. 

We claim that every amenable subgroup $H \leq \Gamma$ is virtually free abelian of rank $k$, with $k \leq n$. In particular the group $H$ is finitely generated, and since the group $\Gamma$ is countable, this implies that there are only countably many amenable subgroups. To prove the claim, we let $H$ be an amenable subgroup of $\Gamma$. Since $H$ does not contain non-abelian free subgroups, the projection of the action of $H$ on each $T_i$ must fix a vertex or a point in $\partial T_i$. If $r \geq 0$ is the number of trees for which the first situation does not happen, then we have a natural morphism from $H$ to $\mathbf{Z}^r$, whose kernel is denoted by $K$. Every finitely generated subgroup of $K$ must have a fixed point in each $T_i$. By assumption $\Gamma$ must act properly on the product of trees, so it follows that $K$ is locally finite, and therefore finite thanks to the assumption on $\Gamma$. So $H$ is finite-by-(free abelian of finite rank), and the proof of the claim is complete.
\end{proof}

\section{Piecewise prescribed tree automorphisms} \label{sec-Pw}

In this paragraph we consider a \enquote{piecewise-ation} process for subgroups of $\aut$, and show how this automatically provides examples of countable groups that are not $C^\ast$-simple.

\bigskip

Let $A$ be a finite subtree of $T$, and let $v_1,\ldots,v_n$ be the vertices of $A$ having at least one neighbour that is not in $A$. For $i=1,\ldots,n$, we denote by $T_i$ the subtree of $T$ made of vertices whose projection on $A$ is the vertex $v_i$. By construction the subtrees $T_i$ are disjoint, and every vertex of $T$ that is not in $A$ lies in some $T_i$. With a slight abuse of notation, we will write $T \setminus A = \sqcup_{i=1}^n T_i$, although every $v_i$ lies in both $T_i$ and $A$. 

\begin{defi}
For a subgroup $G \leq \aut$, we denote by $\Pg$ the set of automorphisms of $T$ acting piecewise like $G$, i.e.\ the set of $\gamma \in \aut$ so that there exists a finite subtree $A$ of $T$ such that, if $T \setminus A = \sqcup_{i=1}^n T_i$, then for every $i$ there exists $g_i \in G$ such that $\gamma$ and $g_i$ coincide on $T_i$.
\end{defi}

We leave to the reader the verification that $\Pg$ is indeed a subgroup of $\aut$, which is countable as soon as $G$ and $T$ are countable. When the tree $T$ is locally finite, $\Pg$ coincides with the intersection in $\mathrm{Homeo}(\partial T)$ of the topological full group associated to $G \curvearrowright \partial T$ and the group $\aut$.

\bigskip

Recall that if $\mathcal{P}$ is a property of groups, we say that a group is locally $\mathcal{P}$ if every finitely generated subgroup has $\mathcal{P}$. 

\begin{lem} \label{lem-fixator-amenable}
Let $\mathcal{P}$ be a property of groups stable by taking subgroups, and $G \leq \aut$ a subgroup such that fixators of edges in $G$ have $\mathcal{P}$. Then for every $\xi \in \partial T$, the stabilizer of $\xi$ in $G$ is (locally $\mathcal{P}$)-by-$\mathbf{Z}$ or locally $\mathcal{P}$.
\end{lem}

\begin{proof}
Let $(v_0,v_1,\ldots)$ be a geodesic ray representing the end $\xi$. The group $G_{\xi}$ admits a normal subgroup $G_{\xi}^0$ consisting of elements $g \in G$ fixing an infinite subray $(v_n,v_{n+1},\ldots)$, where $n \geq 0$ depends on $g$. Moreover the quotient of $G_{\xi}$ by $G_{\xi}^0$ is either infinite cyclic or trivial, according to whether there exists a hyperbolic isometry in $G$ having $\xi$ as a fixed point. Therefore it is enough to show that every finitely generated subgroup of $G_{\xi}^0$ has $\mathcal{P}$. Now for such a finitely generated subgroup $K$, there is an $N \geq 1$ such that the entire $K$ fixes the subray $(v_N,v_{N+1},\ldots)$. In particular $K$ lies inside the fixator of an edge in $G$. The latter has $\mathcal{P}$ by assumption and $\mathcal{P}$ goes to subgroups, so the proof is complete.
\end{proof}

\begin{lem} \label{lem-P-stable}
Let $\mathcal{P}$ be a property of groups stable by taking subgroups, quotients and extensions. Let $G$ be a group generated by a family of normal subgroups $(N_i)_{i \in I}$ together with a subgroup $H$. If all the $N_i$ and $H$ have $\mathcal{P}$, then $G$ is locally $\mathcal{P}$.
\end{lem}

\begin{proof}
Observe that a subgroup $K$ generated by $H$ and some normal subgroup $N \triangleleft G$ with $\mathcal{P}$ must be an extension of $N$ by $H/(H \cap N)$. These have $\mathcal{P}$ by assumption, so $K$ has $\mathcal{P}$ as well.

Now every finitely generated subgroup of $G$ lies in some subgroup of the form $N_{i_1} \ldots N_{i_k} H$ because all the $N_i$ are normal. An easy induction shows that every $N_{i_1} \ldots N_{i_k} H$ has $\mathcal{P}$ thanks to the previous observation, so the statement is proved.
\end{proof}

Assume that $\mathcal{P}$ is a property of groups with the following properties: 
\begin{itemize}
\item $\mathcal{P}$ is stable by taking subgroups, quotients and extensions;
\item a group has $\mathcal{P}$ if and only if all its finitely generated subgroups have $\mathcal{P}$ (i.e.\ $\mathcal{P}$ = locally $\mathcal{P}$);
\item every group that is (locally finite)-by-$\mathcal{P}$ has $\mathcal{P}$. 
\end{itemize}

Examples of such properties are amenability, elementary amenability or being locally finite. 

\begin{prop} \label{prop-pw-amenable}
Let $G \leq \aut$ such that fixators of vertices in $G$ have $\mathcal{P}$. Then fixators of vertices in $\Pg$ also have $\mathcal{P}$.

In particular fixators of ends in $\Pg$ are $\mathcal{P}$-by-$\mathbf{Z}$ or $\mathcal{P}$.
\end{prop}

\begin{proof}
For every vertex $v$ and every integer $n \geq 0$, let $K_{n}(v)$ be the set of elements $\gamma \in \Pg_{v}$ such that there exists a finite subtree $A$ inside the ball of radius $n$ around $v$, such that $\gamma$ coincides with an element of $G$ on each component of $T \setminus A$. It is not hard to check that, for a fixed vertex $v$, each $K_n(v)$ is a subgroup, and the sequence $(K_n(v))$ is increasing and ascends to $\Pg_{v}$.  

We fix some vertex $v_0$, and we prove that $\Pg_{v_0}$ has $\mathcal{P}$. Since $\mathcal{P}$ is a local property, by the previous observation it is enough to prove that each $K_n(v_0)$ has $\mathcal{P}$. Let us prove this fact by induction. By definition $K_0(v_0)$ is equal to the stabilizer of $v_0$ in $G$. The latter has $\mathcal{P}$ by assumption, so the result holds for $n=0$.

Now we assume that $n \geq 1$ is so that $K_{n-1}(v)$ has $\mathcal{P}$ for every vertex $v$, and we prove that $K_{n}(v_0)$ has $\mathcal{P}$. We denote by $(v_i)_{i \in I}$ the set of neighbours of $v_0$, and by $T_i$ the unique half-tree containing $v_0$ but not $v_i$. For $i \in I$, we let $F_n(v_i) = \Pg_{T_i} \cap K_n(v_i)$: this is the set of automorphisms $\gamma$ fixing the half-tree $T_i$ and such that there exists a finite subtree $A$ inside the ball of radius $n$ around $v_i$, such that $\gamma$ coincides with an element of $G$ on each component of $T \setminus A$.

Let us denote by $\pi: \Pg_{v_0} \rightarrow \mathrm{Sym}(I)$ the morphism coming from the action of the group $\Pg_{v_0}$ on the set of neighbours of $v_0$. The kernel of $\pi$ is the set of elements of $\Pg$ fixing the $1$-ball around the vertex $v_0$. If we let $N_{n}(v_0) = K_{n}(v_0) \cap \ker \pi$, then we easily see that $N_{n}(v_0)$ is generated by all the subgroups $F_{n-1}(v_i)$ together with the subgroup $G_{B(v_0,1)}$ of $G$ fixing the $1$-ball around $v_0$. By the induction hypothesis, every $K_{n-1}(v_i)$ has $\mathcal{P}$, so a fortiori every $F_{n-1}(v_i)$ has $\mathcal{P}$. Moreover $G_{B(v_0,1)}$ also has $\mathcal{P}$ (as a subgroup of $G_{v_0}$), and every $F_{n-1}(v_i)$ is normal in $N_{n}(v_0)$. Therefore it follows from Lemma \ref{lem-P-stable} that the subgroup $N_{n}(v_0)$ is locally $\mathcal{P}$, and hence is $\mathcal{P}$. 

Moreover it readily follows from the definition of the group $\Pg$ that the image of $\Pg_{v_0}$ in $\mathrm{Sym}(I)$ lies inside the subgroup $\mathrm{Sym}_0(I) \pi (G_{v_0})$ of $\mathrm{Sym}(I)$, where $\mathrm{Sym}_0(I)$ is the group of finitary permutations of $I$. The group $G_{v_0}$ has $\mathcal{P}$, so its image $\pi (G_{v_0})$ also has $\mathcal{P}$. By assumption an extension of a locally finite group by a group with $\mathcal{P}$ remains $\mathcal{P}$, so it follows that $\mathrm{Sym}^0(I) \pi (G_{v_0})$ (and therefore $\pi(K_{n}(v_0))$ has $\mathcal{P}$. 

So we have proved that the group $K_{n}(v_0)$ is an extension of groups with $\mathcal{P}$. By assumption $\mathcal{P}$ is stable under extension, so $K_{n}(v_0)$ must have $\mathcal{P}$. This finishes the induction step and terminates the proof of the first statement.

To obtain the second statement, remark that fixators of edges in $\Pg$ must also have $\mathcal{P}$, and apply Lemma \ref{lem-fixator-amenable}.
\end{proof}

In particular when applying Proposition \ref{prop-pw-amenable} with $\mathcal{P}$ equals amenability, we obtain the following result.

\begin{cor} \label{cor-pg-amenable}
Let $G \leq \aut$ be a subgroup such that fixators of vertices in $G$ are amenable. Then for every $\xi \in \partial T$, the stabilizer of $\xi$ in $\Pg$ is amenable.
\end{cor}

We now prove Theorem \ref{thm-pw-c*simple} from the introduction.

\begin{proof}[Proof of Theorem \ref{thm-pw-c*simple}]
(a). Let $\Gamma \leq \Pg$ containing $G$ and having non-trivial fixators of half-trees. It is clear that the action of $\Gamma$ on $T$ is minimal and of general type because it is already the case for $G$ and $\Gamma$ contains $G$.

Since fixators of vertices in $G$ are amenable, by Corollary \ref{cor-pg-amenable} we obtain that stabilizers of ends in $\Pg$ are amenable, so a fortiori the same is true in $\Gamma$. Moreover fixators of half-trees in $\Gamma$ are non-trivial by assumption, so Theorem \ref{thm-general-c*-intro} gives the conclusion. 

(b). Remark that if the action of a group on $T$ is minimal and of general type, the non-triviality of the fixator of \textit{one} half-tree is equivalent to the non-triviality of all fixators of half-trees. By combining this observation with statement (a), we see that it is enough to exhibit some non-trivial element in $\Pg$ fixing a half-tree.

Since stabilizers of vertices in $G$ are non-trivial, we may find a vertex $v$ of degree at least three, two different edges $e_1$ and $e_2$ around $v$ and $g \in G_v$ such that $g(e_1)=e_2$. Denote by $T_1$ (resp.\ $T_2$) the half-tree emanating from $e_1$ (resp.\ $e_2$) and not containing $v$, so that $T_1$ and $T_2$ are disjoint and $g(T_1)=T_2$. Now consider $\gamma$ acting like $g$ on $T_1$, like $g^{-1}$ on $T_2$ and being the identity elsewhere. By construction the element $\gamma$ is non-trivial, $\gamma$ fixes a half-tree because the degree of $v$ is at least three, and $\gamma$ clearly belongs to $\Pg$. This terminates the proof.
\end{proof}

We now explain how Theorem \ref{thm-pw-c*simple} allows to construct a multitude of groups with trivial amenable radical that are not $C^\ast$-simple. First start with:
\begin{itemize}
\item an amalgamated product $G = A\ast_C B$, where $A,B$ are countable amenable groups, $C$ is a proper subgroup of both $A$ and $B$ that is not of index two in both $A$ and $B$; or
\item an HNN-extension $G = \mathrm{HNN}(H,K,L,\varphi)$, where $H$ is a countable amenable group, $K,L$ are proper subgroups, and $\varphi: K \rightarrow L$ is an isomorphism;
\end{itemize}
and consider the action of $G$ on its Bass-Serre tree $T$. This action is minimal and of general type, and fixators of vertices are amenable. Therefore if we denote by $\hat{G}$ the image of $G \rightarrow \aut$, Theorem \ref{thm-pw-c*simple} shows that $\mathrm{Pw}(\hat{G})$ has trivial amenable radical and is not $C^\ast$-simple. Note that the kernel of $G \rightarrow \hat{G}$ can be explicitly computed (see for instance \cite[\S 5.1 and 5.2]{dlH-P}).

For example one may take for $A,B$ two finite groups of cardinality $p$ and $q$ and $C=1$, with $(p,q) \neq (2,2)$. In this situation $G = A \ast B$ acts on a $(p,q)$-biregular tree. For $(p,q)=(2,3)$, the group $G$ is isomorphic to $\mathrm{PSL}(2,\mathbf{Z})$, and we obtain that the group of automorphisms of the $(2,3)$-biregular tree which are piecewise $\mathrm{PSL}(2,\mathbf{Z})$ is not $C^\ast$-simple. Interestingly, one may check that this group is isomorphic to $G(F,F')$ with $F = \mathrm{Alt}(3)$ and $F' = \mathrm{Sym}(3)$ (see Section \ref{sec-G(F,F')} for the definition of these groups). 

The case of HNN-extensions includes for example the Baumslag-Solitar group $\mathrm{BS}(m,n) = \left\langle t,x \, | \, tx^{m}t^{-1}  = x^n\right\rangle$ with $|n| > |m| \geq 2$, whose Bass-Serre tree $T$ is regular of degree $|m|+|n|$. While $\mathrm{BS}(m,n)$ is known to be $C^\ast$-simple \cite{dlH-P}, Theorem \ref{thm-pw-c*simple} shows that the group $\mathrm{Pw}(\mathrm{BS}(m,n))$ is \textit{not} $C^\ast$-simple. Note that the recent work \cite{SvR} shows that some Schlichting completion of $\mathrm{BS}(m,n)$, which coincides with the closure of $\mathrm{BS}(m,n)$ in $\aut$, is $C^\ast$-simple. 

\section{Groups with prescribed local action} \label{sec-G(F,F')}

Let $\Omega$ be a set (with no further assumption) of cardinality at least three, and $T_\Omega$ a regular tree of degree the cardinality of $\Omega$. We fix a coloring of the edges $c : E(T_\Omega) \rightarrow \Omega$ such that for every vertex $v$, the restriction of $c$ to the set of edges containing $v$ induces a bijection with $\Omega$. For every $g \in \autd$ and every vertex $v$, the action of $g$ around $v$ gives rise to a permutation $\sg \in \Sy$ after identification of the edges around $v$ and around $g(v)$ with $\Omega$.

Given a permutation group $F \leq \Sy$, we let $U(F) \leq \autd$ be the subgroup consisting of elements $g$ such that $\sg \in F$ for all vertices $v$. It is not hard to check that the action of the group $U(F)$ on $T_\Omega$ is always transitive on vertices and of general type. 

When $\Omega$ is finite, the groups $U(F)$ are important closed subgroups of $\autd$, and play an essential role in the main construction of \cite{BM-IHES-2}. The study of these groups in the setting of a non-locally finite tree has been initiated in \cite{Smith}.

Now given a second permutation group $F' \leq \Sy$ containing $F$, we define $G(F,F') \leq \autd$ as the set of automorphisms $g$ such that $\sg \in F'$ for all $v$, and $\sg \in F$ for all but finitely many $v$. Clearly $G(F,F')$ contains $U(F)$, and it is a simple verification that $G(F,F')$ is always a group.

\begin{rmq}
When $\Omega$ is a finite set, it readily follows from the definition that the group $\mathrm{Pw}(U(F))$ coincides with $G(F,\Sy)$. When $\Omega$ is infinite, we always have $\mathrm{Pw}(U(F)) \leq G(F,\Sy)$, but this inclusion is strict in general, so that the groups investigated in this section are not covered by Section \ref{sec-Pw}.
\end{rmq}

\subsection{Proof of Theorem C}

The aim of this paragraph is to show that, under mild assumptions on the permutation groups $F,F'$, the group $G(F,F')$ is not $C^\ast$-simple.

\begin{lem} \label{lem-fixator-F-freely}
Assume that the permutation group $F$ acts freely on $\Omega$. Then fixators of edges in $U(F)$ are trivial.
\end{lem}

\begin{proof}
Let $e$ be an edge of $T_\Omega$ and let $a \in \Omega$ be its color. Assume that $g \in U(F)$ fixes $e$. If $v$ is one of the two vertices of $e$, then the permutation $\sigma(g,v)$ fixes $a$ because $g$ fixes the edge $e$, and $\sigma(g,v) \in F$ because $g \in U(F)$. By assumption $F$ acts freely on $\Omega$, so we obtain that $\sigma(g,v)$ is trivial. Therefore $g$ fixes the $1$-ball around the edge $e$, and by repeating the argument we immediately obtain that $g$ must be trivial.
\end{proof}

For a subgroup $G \leq \autd$, we will denote by $G^\ast$ the subgroup of $G$ of index at most two preserving the natural bipartition of vertices of $T_\Omega$.

\begin{lem} \label{lem-torsion-free}
Assume that the permutation group $F$ is torsion free. Then the group $U(F)^\ast$ is torsion free.
\end{lem}

\begin{proof}
Any element $g \in U(F)^\ast$ is either hyperbolic or fixes a vertex. In the former case $g$ is clearly of infinite order, so we may assume that there exists a vertex $v$ such that $g(v)=v$. Assume that $g^n=1$ for some $n \geq 1$. Since $g$ fixes $v$, one has $\sigma(g,v)^n=\sigma(g^n,v)=1$, and therefore $\sigma(g,v)=1$ since $F$ is torsion free. An easy induction on the distance between a vertex $w$ and $v$ shows that $\sigma(g,w)=1$ for every vertex $w$, and we deduce that $g$ must be trivial. 
\end{proof}

Under the assumption that $F$ acts freely on $\Omega$, the following result relates the amenability of fixators of half-trees in $G(F,F')$ to the amenability of point stabilizers $F_a'$ in $F'$, $a \in \Omega$. We emphasize that, for fixators of half-trees in $G(F,F')$ to the amenable, the permutation group $F'$ (and a fortiori vertex stabilizers in $G(F,F')$) need \textit{not} be amenable (note the difference with Proposition \ref{prop-pw-amenable}). This will be a crucial point in the proof of Theorem \ref{thm-intro-embedding}.

\begin{prop} \label{prop-ht-amenable}
Let $\Omega$ be a set, and $F \leq F' \leq \Sy$ permutation groups such that $F$ acts freely on $\Omega$ and $F'$ preserves the orbits of $F$. Then fixators of half-trees in $G(F,F')$ are amenable if and only if $F_a'$ is amenable for every $a \in \Omega$.
\end{prop}

\begin{proof}
We assume that $F_a'$ is amenable for every $a \in \Omega$, and we prove that fixators of half-trees in $G(F,F')$ are amenable. 

Given a half-tree $T$ in $T_\Omega$, we denote by $e_{T}$ the edge defining $T$, and by $v_{T}$ the vertex of $e_{T}$ that does not belong to $T$. For every $n \geq 0$, we let $K_{n,T}$ be the set of elements $g \in G(F,F')$ fixing $T$ and such that $\sigma(g,w) \in F$ for every vertex $w$ at distance at least $n$ from $v_T$. For every $b \in \Omega$, we denote by $T_b$ the half-tree containing $v_T$ and defined by the edge around $v_T$ having color $b$.

Let us denote by $a$ the color of the edge $e_T$. The action of the group $K_{n+1,T}$ around the vertex $v_{T}$ yields a morphism from $K_{n+1,T}$ to $F_a'$, and one may check that this morphism is onto thanks to the assumption that $F'$ preserves the orbits of $F$ (see \cite[Lemma 3.4]{LB-ae}). Moreover the kernel of this morphism is the fixator of the $1$-ball around $v_{T}$ in $K_{n+1,T}$, so that we have a short exact sequence \begin{equation} \label{eq-ses} 1 \rightarrow \bigoplus_{b \neq a} K_{n,T_b} \rightarrow K_{n+1,T} \rightarrow F_a' \rightarrow 1 \end{equation} for every $n \geq 0$.

For a fixed half-tree $T$, the sequence $(K_{n,T})$ is increasing and ascends to $G(F,F')_{T}$. So to prove that $G(F,F')_{T}$ is amenable, it is enough to prove that each $K_{n,T}$ is amenable. We shall prove by induction on $n \geq 0$ that $K_{n,T}$ is amenable for every half-tree $T$.

Assume that $n=0$ and let $T$ be a half-tree in $T_\Omega$. By definition the group $K_{0,T}$ lies inside $U(F)$ and fixes an edge. Since $F$ acts freely on $\Omega$, according to Lemma \ref{lem-fixator-F-freely} this implies that $K_{0,T}$ is trivial. So the result holds for $n=0$. 

Now assume that $n \geq 0$ is such that $K_{n,T}$ is amenable for every half-tree $T$. We let $T$ be a half-tree and we show that $K_{n+1,T}$ is amenable. Denote by $a$ the color of $e_T$. Combining the short exact sequence (\ref{eq-ses}) with the assumption on $n$, we see that the group $K_{n+1,T}$ is an extension of an amenable group by the group $F_{a}'$. By assumption all point stabilizers in $F'$ are amenable, so it follows that $K_{n+1,T}$ is an extension of amenable groups, and therefore is amenable. This terminates the proof of one implication.

The converse implication is clear because point stabilizers in $F'$ embed in every fixator of half-tree.
\end{proof}

\begin{cor} \label{cor-stab-xi-g(f,f')}
Let $\Omega$ be a set, and $F \leq F' \leq \Sy$ permutation groups such that $F$ acts freely on $\Omega$ and $F'$ preserves the orbits of $F$. Assume that $F_a'$ is amenable for every $a \in \Omega$. Then for every $\xi \in \partial T_\Omega$, the fixator of $\xi$ in $G(F,F')$ is amenable.
\end{cor}

\begin{proof}
Let $e$ be an edge of $T_\Omega$. It is not hard to check that the fixator of $e$ in $G(F,F')$ is the product of the fixators of the two half-trees defined by $e$ (in other words, the group $G(F,F')$ has the edge-independence property in the terminology of \cite{LB-ae}). According to Proposition \ref{prop-ht-amenable}, these fixators of half-trees are amenable, so the fixator of $e$ in $G(F,F')$ is also amenable. The statement then follows by applying Lemma \ref{lem-fixator-amenable}.
\end{proof}

When the set $\Omega$ is countable and $F$ acts freely on $\Omega$, it immediately follows from Lemma \ref{lem-fixator-F-freely} that the group $U(F)$ is countable. So if moreover the permutation group $F'$ is also countable, it is not hard to deduce that $G(F,F')$ is a countable group.

Before giving the proof of Theorem \ref{thm-c*-g(f,f')}, let us mention that the argument also applies to the group $G(F,F')^\ast$, which is therefore not $C^\ast$-simple either.

\begin{proof}[Proof of Theorem \ref{thm-c*-g(f,f')}]
By definition $G(F,F')$ always contains the group $U(F)$, whose action on $T_\Omega$ is minimal and of general type, so a fortiori the same is true for $G(F,F')$.

Let us prove that fixators of half-trees in $G(F,F')$ are non-trivial. Let $T$ be a half-tree in $T_\Omega$. We point out that $G(F,F')_{T}$ is actually infinite, but this is not needed here, and we shall give a simple argument to exhibit a non-trivial element in $G(F,F')_{T}$. We let $e$ be the edge defining $T$, and denote by $a  \in \Omega$ the color of $e$. We also let $T'$ be the half-tree facing $T$, and $v$ be the vertex of $e$ that belongs to $T'$. Since $F$ is a proper subgroup of $F'$ and $F'$ stabilizes the orbits of $F$, there is a non-trivial $\sigma \in F'$ such that $\sigma(a)=a$. For every $b \in \Omega$, we let $\sigma_b \in F$ such that $\sigma_b(b) = \sigma(b)$. Consider the automorphism $g$ of $T_\Omega$ fixing the edge $e$ and such that:
\begin{itemize}
\item [-] $\sigma(g,w) = id$ for every vertex $w$ in $T$;
\item [-] $\sg = \sigma$;
\item [-] $\sigma(g,w) = \sigma_b$ for every vertex $w \neq v$ of $T'$, where $b$ is the color of the unique edge emanating from $v$ towards $w$.
\end{itemize}
It is a simple verification that the automorphism $g$ is well defined. By construction $g$ fixes $T$ but is not trivial because $\sigma \neq 1$, and $g \in G(F,F')$ because $\sigma(g,w) \in F$ for every $w \neq v$.

Now since point stabilizers in $F'$ are supposed to be amenable, it follows from Corollary \ref{cor-stab-xi-g(f,f')} that stabilizers of ends in $G(F,F')$ are amenable. Therefore we are in position to apply Theorem \ref{thm-general-c*-intro}, which gives the conclusion.
\end{proof}

\subsection{Embedding theorem}

We now prove Theorem \ref{thm-intro-embedding}. Our construction is extremely flexible: it admits a parameter that is a countable amenable group $A$, so that when $A$ varies, we actually obtain many embeddings with the required properties.

We let $\Gamma$ be a countable group. For an arbitrary group $A$, we consider the wreath product $F' = \Gamma \wr A$, and we let $F = \Gamma^{(A)}$ be the kernel of the natural morphism from $F'$ onto $A$. If $\Omega = F' / A$, then the group $F$ acts transitively on $\Omega$. Clearly the stabilizer of the coset $A$ in $F$ is trivial by construction, and since all point stabilizers in $F$ are conjugate, we obtain that $F$ acts freely on $\Omega$. 

We claim that the action of $F'$ on $\Omega$ is always faithful when $\Gamma$ is non-trivial. Indeed, given $1 \neq \gamma \in \Gamma$ and $\alpha \in A$, let us consider the element $f \in F = \Gamma^{(A)}$ equals to the identity everywhere except at $a$, where it takes the value $\gamma$. Now assume that $\beta \in A \cap f A f^{-1}$. Then an easy computation shows that one must have $f = \beta f \beta^{-1}$. Therefore these two elements have the same support, which shows that $\beta$ must be trivial. So there are two conjugates of $A$ in $F'$ that intersect trivially, which implies that the action of $F'$ on $\Omega = F' / A$ is faithful. 

Clearly $F'$ (and therefore $\Gamma)$ embeds into any vertex stabilizer in $G(F,F')$, so in particular into $G(F,F')^\ast$.

If we assume in addition that $A$ is countable, amenable and non-trivial, then $F'$ is also countable, and its point stabilizers are amenable because these are conjugates of $A$. Therefore we may apply Theorem \ref{thm-c*-g(f,f')}, which shows that the countable group $G(F,F')^\ast$ has trivial amenable radical and is not $C^\ast$-simple.

Now if we assume that $\Gamma$ is finitely generated, then $F'$ can easily me made finitely generated as well (by choosing $A$ with the same property). Since $F$ is simply transitive, the group $G(F,F')^\ast$ is generated by two copies of $F'$ (see Corollary 3.10 and Remark 3.11 in \cite{LB-ae}. The argument is given there for $\Omega$ finite, but the proof works verbatim in our setting). In particular this shows that $G(F,F')^\ast$ is finitely generated.

Finally if $\Gamma$ is torsion free, then by choosing $A$ torsion free (e.g.\ $A=\mathbf{Z}$) we see that $F'$ is also torsion free. According to Lemma \ref{lem-torsion-free}, this implies that the group $U(F')^\ast$ has no torsion elements, and therefore the same is true in $G(F,F')^\ast \leq U(F')^\ast$.

\nocite{*}
\bibliographystyle{amsalpha}
\bibliography{c-simple}

\providecommand{\bysame}{\leavevmode\hbox to3em{\hrulefill}\thinspace}
\providecommand{\MR}{\relax\ifhmode\unskip\space\fi MR }
\providecommand{\MRhref}[2]{%
  \href{http://www.ams.org/mathscinet-getitem?mr=#1}{#2}
}
\providecommand{\href}[2]{#2}
\begin{thebibliography}{BKKO14}

\bibitem[AL80]{Ake-Lee}
C.~Akemann and T.~Lee, \emph{Some simple {$C^{\ast} $}-algebras associated with
  free groups}, Indiana Univ. Math. J. \textbf{29} (1980), no.~4, 505--511.

\bibitem[AM07]{A-M}
G.~Arzhantseva and A.~Minasyan, \emph{Relatively hyperbolic groups are
  {$C^\ast$}-simple}, J. Funct. Anal. \textbf{243} (2007), no.~1, 345--351.

\bibitem[BCH94]{B-C-dlH}
B.~Bekka, M.~Cowling, and P.~de~la Harpe, \emph{Some groups whose reduced
  {$C^*$}-algebra is simple}, Inst. Hautes \'Etudes Sci. Publ. Math. (1994),
  no.~80, 117--134.

\bibitem[B{\'e}d91]{Bed-disgr}
E.~B{\'e}dos, \emph{Discrete groups and simple {$C^*$}-algebras}, Math. Proc.
  Cambridge Philos. Soc. \textbf{109} (1991), no.~3, 521--537.

\bibitem[BH00]{B-dlH}
B.~Bekka and P.~de~la Harpe, \emph{Groups with simple reduced
  {$C^*$}-algebras}, Expo. Math. \textbf{18} (2000), no.~3, 215--230.

\bibitem[BH04]{Br-dlH}
M.~Bridson and P.~de~la Harpe, \emph{Mapping class groups and outer
  automorphism groups of free groups are {$C^*$}-simple}, J. Funct. Anal.
  \textbf{212} (2004), no.~1, 195--205.

\bibitem[BKKO14]{BKKO}
E.~Breuillard, M.~Kalantar, M.~Kennedy, and N.~Ozawa,
  \emph{{$C^\ast$}-simplicity and the unique trace property for discrete
  groups}, arXiv:1410.2518v2 (2014).

\bibitem[BM00]{BM-IHES-2}
M.~Burger and S.~Mozes, \emph{Lattices in product of trees}, Inst. Hautes
  \'Etudes Sci. Publ. Math. (2000), no.~92, 151--194.

\bibitem[DGO11]{DGO}
F.~Dahmani, V.~Guirardel, and D.~Osin, \emph{Hyperbolically embedded subgroups
  and rotating families in groups acting on hyperbolic spaces}, arXiv:1111.7048
  (2011).

\bibitem[Har88]{Har-88}
P.~de~la Harpe, \emph{Groupes hyperboliques, alg\`ebres d'op\'erateurs et un
  th\'eor\`eme de {J}olissaint}, C. R. Acad. Sci. Paris S\'er. I Math.
  \textbf{307} (1988), no.~14, 771--774.

\bibitem[Har07]{dlH-survey}
\bysame, \emph{On simplicity of reduced {$C^\ast$}-algebras of groups}, Bull.
  Lond. Math. Soc. \textbf{39} (2007), no.~1, 1--26.

\bibitem[HO14]{H-O}
U.~Haagerup and K.K. Olesen, \emph{On conditions towards the non-amenability of
  {R}ichard {T}hompson's group {F}}, preprint (2014).

\bibitem[HP11]{dlH-P}
P.~de~la Harpe and J-P. Pr{\'e}aux, \emph{{$C^*$}-simple groups: amalgamated
  free products, {HNN}-extensions, and fundamental groups of 3-manifolds}, J.
  Topol. Anal. \textbf{3} (2011), no.~4, 451--489.

\bibitem[KK14]{KK}
M.~Kalantar and M.~Kennedy, \emph{Boundaries of reduced {$C^\ast$}-algebras of
  discrete groups}, To appear in J. Reine Angew. Math. (arXiv:1405.4359v3)
  (2014).

\bibitem[KS16]{Kar-Sag}
A.~Kar and M.~Sageev, \emph{Ping {P}ong on {CAT}(0) cube complexes}, Comment.
  Math. Helv. \textbf{91} (2016), no.~3, 543--561.

\bibitem[LB16]{LB-ae}
A.~Le~Boudec, \emph{Groups acting on trees with almost prescribed local
  action}, Comment. Math. Helv. \textbf{91} (2016), no.~2, 253--293.

\bibitem[Mon13]{Monod}
N.~Monod, \emph{Groups of piecewise projective homeomorphisms}, Proc. Natl.
  Acad. Sci. USA \textbf{110} (2013), no.~12, 4524--4527.

\bibitem[MS04]{Mon-Sha}
N.~Monod and Y.~Shalom, \emph{Cocycle superrigidity and bounded cohomology for
  negatively curved spaces}, J. Differential Geom. \textbf{67} (2004), no.~3,
  395--455.

\bibitem[OO14]{O-O}
A.~Olshanskii and D.~Osin, \emph{{$C^*$}-simple groups without free subgroups},
  Groups Geom. Dyn. \textbf{8} (2014), no.~3, 933--983.

\bibitem[Oza14]{Ozawa-notes}
N.~Ozawa, \emph{Lecture on the {F}urstenberg boundary and
  {$C^\ast$}-simplicity},
  http://www.kurims.kyoto-u.ac.jp/~narutaka/notes/yokou2014.pdf (2014).

\bibitem[Pow75]{Pow}
R.~Powers, \emph{Simplicity of the {$C^{\ast}$}-algebra associated with the
  free group on two generators}, Duke Math. J. \textbf{42} (1975), 151--156.

\bibitem[Poz08]{Poz}
T.~Poznansky, \emph{Characterization of linear groups whose reduced
  {$C^*$}-algebras are simple}, arXiv:0812.2486v7 (2008).

\bibitem[PS79]{Pas-Sal}
W.~Paschke and N.~Salinas, \emph{{$C^{\ast}$}-algebras associated with free
  products of groups}, Pacific J. Math. \textbf{82} (1979), no.~1, 211--221.

\bibitem[PV91]{Pa-Va}
I.~Pays and A.~Valette, \emph{Sous-groupes libres dans les groupes
  d'automorphismes d'arbres}, Enseign. Math. (2) \textbf{37} (1991), no.~1-2,
  151--174.

\bibitem[Rau15]{SvR}
S.~Raum, \emph{C*-simplicity of locally compact {P}owers groups}, To appear in
  J. Reine Angew. Math. (arXiv:1505.07793) (2015).

\bibitem[Smi14]{Smith}
S.~Smith, \emph{A product for permutation groups and topological groups},
  arXiv:1407.5697v1 (2014).

\end{thebibliography}

\end{document}